\documentclass{article}
\usepackage{amsthm}
\usepackage{amsfonts}
\usepackage{amssymb,latexsym}
\usepackage{amsbsy}
\usepackage{amsxtra}
\usepackage{hyperref}
\usepackage{amsmath}
\usepackage{mathabx}
\usepackage{mathrsfs}
\usepackage{graphicx}
\usepackage{bbm}
\usepackage{geometry}
\usepackage{subfig}
\usepackage{caption}

\theoremstyle{plain}
\newtheorem{theorem}{Theorem}
\newtheorem{lemma}{Lemma}
\newtheorem{proposition}{Proposition}
\newtheorem{corollary}{Corollary}

\begin{document}
 \title{Resolvent Estimates for Schr\"{o}dinger Operators with Potentials in Lebesgue Spaces}
 \author{Tianyi Ren}

\maketitle
 \begin{abstract}
 We prove resolvent estimates in the Euclidean setting for Schr\"{o}dinger operators with potentials in Lebesgue spaces: $-\Delta+V$. The $(L^{2}, L^{p})$ estimates were already obtained by Blair-Sogge-Sire, but we extend their result to other $(L^{p}, L^{q})$ estimates using their idea and the result and method of Kwon-Lee on non-uniform resolvent estimates in the Euclidean space. 
 \end{abstract}

\maketitle
\section{Introduction}

Resolvent estimates for Schr\"{o}dinger operators has always been of keen interest in the study of harmonic analysis and partial differential equations over the past 30 years or so. This interest originated from the consequential work of Kenig, Ruiz and Sogge \cite{KRS} in 1986 that states:

\vspace{0.8cm}
\begin{itshape}
\noindent On $\mathbb{R}^n$ where $n\geqslant3$, if the pair of exponents $r$ and $s$ satisfy the conditions
	
	(a)\[\frac{1}{r}-\frac{1}{s}=\frac{2}{n},\]
	
	(b)\[\mathrm{min}\Big\{\Big|\frac{1}{r}-\frac{1}{2}\Big|, \Big|\frac{1}{s}-\frac{1}{2}\Big|\Big\}> \frac{1}{2n},\]
then there exists a constant $C$, depending only on $n$, $r$ and $s$, such that the following inequality holds:
	\[\|u\|_{L^{s}(\mathbb{R}^n)}\leqslant C\|(\Delta+z)u\|_{L^{r}(\mathbb{R}^n)}, \quad u\in H^{2, r}(\mathbb{R}^n), \quad z\in \mathbb{C}.\]
\end{itshape}

Since then, efforts have branched into several different directions. A most notable one is the corresponding inequality on manifolds, where the Laplacian operator is replaced by the Laplace-Beltrami operator. In this branch we mention the pioneering work of Dos Santos Ferreira-Kenig-Salo \cite{FKS}, and also the prominent work of Bourgain-Sogge-Shao-Yao \cite{BSSY}, of Shao-Yao \cite{SY} and of Huang-Sogge \cite{HS}. In recent years, Frank-Schimmer \cite{FS} proved the endpoint case of Ferreira-Kenig-Salo \cite{FKS}'s result, and their rediscovery of a method in Gutierrez \cite{Gutierrez} has inspired a few later works. 

Another important direction is to prove the resolvent inequality for other pairs of exponents. It was shown in \cite{KRS} that the conditions (a) and (b) are necessary to obtain a uniform bound that does not depend on the complex number $z$. On a coordinate plane whose two axes represent $\frac{1}{p}$ and $\frac{1}{q}$ respectively, the pair of exponents $(\frac{1}{p}, \frac{1}{q})$ constitute the open line segment from $(\frac{n+1}{2n}, \frac{n-3}{2n})$ to $(\frac{n-1}{2n}, \frac{n+3}{2n})$. However, for other exponent pairs, an estimate depending on $z$ might also be of interest. In this direction, we point out the early contribution of Guti\'{e}rrez \cite{Gutierrez} to the Euclidean case. In her result, the non-uniform bound is in the form of a power of $z$. Frank \cite{Frank} proved some crucial $(L^{p}, L^{p})$ estimates (he actually obtained stronger bounds in Schatten spaces). Recently, in a marvelous work, Y. Kwon and S. Lee \cite{KL} found the whole region in the coordinate plane for the exponent pair $(\frac{1}{p}, \frac{1}{q})$ where the resolvent operator norm $\|(-\Delta-z)^{-1}\|_{L^{p}\rightarrow L^{q}}$ on $\mathbb{R}^{n}$ is finite for any given $z\in\mathbb{C}\backslash[0, \infty)$, and provide the optimal bounds for the resolvent operator for exponent pairs in almost all this region (as we shall see soon, the two dimensional case is completely solved, and only two dual triangles in three dimensional and up cases remains unsettled).

In addition to the two directions mentioned above, people have also added a potential $V(x)$ to the Laplacian and proved the corresponding resolvent estimates. The setting could be the Euclidean space or manifolds. In 2019, M. Blair, Y. Sire and C. Sogge \cite{BSS} proved that on an $n$ dimensional compact manifold $M$ where $n\geqslant 4$, if $V\in L^{\frac{n}{2}}(M)$, then letting
\[\sigma (p)=\mathrm{min}\big\{n(\frac{1}{2}-\frac{1}{p})-\frac{1}{2}, \frac{n-1}{2}(\frac{1}{2}-\frac{1}{p})\big\},\]
we have for $\lambda\geqslant 1$,
\begin{multline}
\|u\|_{L^{p}(M)}\leqslant C_{p, V}(\lambda^{\sigma(p)-1}\|(-\Delta_{g}+V-(\lambda+i)^{2})u\|_{L^{2}(M)}+\lambda^{\sigma(p)}\|u\|_{L^{2}(M)}),\\
if \quad u\in C^{\infty}(M),
\end{multline}
on condition that $2<p<\frac{2n}{n-3}$.
If one further assumes that $V$ belongs to the Kato class $\mathcal{K}(M)$\footnote{
On manifold $M$, let for $r>0$
 \[ h(r)=\begin{cases}
	|\mathrm{log}r|&\text{if}\quad n=2\\
	r^{2-n}&\text{if}\quad n\geqslant 3.
\end{cases} \]
A function $V(x)$ on $M$ is said to be in the Kato class $\mathcal{K}(M)$ if
\[ \lim\limits_{r\searrow 0}\sup\limits_{x}\int_{B_{r}(x)} h(d_{g}(x, y)|V(y)|dy=0, \]
where $d_{g}(x, y)$ denotes the geodesic distance between $x$ and $y$, $B_{r}(x)$ means the geodesic ball of radius $r$ around $x$, and the integration is with respect to the volume element of the manifold. The Kato class $\mathcal{K}(\mathbb{R}^{n})$ on $\mathbb{R}^{n}$ is defined similarly.
}, they may not only include two and three dimensional cases in their result, but also extend the range of the exponent $p$ to $2<p\leqslant\infty$ in two and three dimensional cases, and to $2<p<\frac{2n}{n-4}$ in higher dimensional cases, and also take $u$ to be in $\mathrm{Dom}(-\Delta_{g}+V)$. Moreover, in higher dimensional cases ($n\geqslant 4$), when $p\in [\frac{2n}{n-4}, \infty]$, the inequality still holds with an additional term related to the spectral projection operator for $\sqrt{-\Delta_{g}+V}$ corresponding to the interval $[2\lambda, \infty)$ added to the right. There is also an analogous result in the Euclidean case, where the only difference we make is that the $V$ has to be in $L^{\frac{n}{2}}(\mathbb{R}^{n})+L^{\infty}(\mathbb{R}^{n})$ to make the crucial idea of the proof works. 

In this paper, we focus on the Euclidean space, and apply the method and result in Kwon-Lee \cite{KL} and the method in Blair-Sire-Sogge \cite{BSS} to prove the inequality as in Blair-Sire-Sogge \cite{BSS} but for a much wider range of exponent pairs $(p, q)$. Although $L^{2}$ spaces are the most important, other exponent pairs are also of use here and there. What is more, we allow the potential to be in other Lebesgue spaces. Before stating our main theorems and their application, we introduce in detail the paper on resolvent estimates of Kwon-Lee \cite{KL}.

\maketitle
\section{Kwon-Lee's Work: Sharp Resolvent Estimates Outside of the Boundedness Range} \label{Section2}

In the coordinate plane, let $I=\{(x, y)\in [0, 1]\times [0, 1], y\leqslant x\}$ for notational convenience in the future. Define
\begin{equation}
	\mathcal{R}_{0}=\begin{cases}
	\{(x, y)\in\mathbb{R}^2: 0\leqslant x, y\leqslant 1, 0\leqslant x-y<1\} & \text{if}\quad n=2,\\
	\{(x, y)\in\mathbb{R}^2: 0\leqslant x, y\leqslant 1, 0\leqslant x-y\leqslant\frac{2}{n}\}\backslash\{(1, \frac{n-2}{n}), (\frac{2}{n}, 0)\} & \text{if}\quad n>2.
	\end{cases}
\end{equation}
This is the region for exponent pairs for which the resolvent operator norm $\|(-\Delta-z)^{-1}\|_{L^{p}\rightarrow L^{q}}$ is finite for any given $z\in\mathbb{C}\backslash[0, \infty)$.

Now we give the sharp bounds for the resolvent operator when $(\frac{1}{p}, \frac{1}{q})$ is in the above just-stated region. For $n\geqslant 2$, given $(\frac{1}{p}, \frac{1}{q})$, define
\begin{equation}
\gamma(p, q):=\mathrm{max}\big\{0, 1-\frac{n+1}{2}\big(\frac{1}{p}-\frac{1}{q}\big), \frac{n+1}{2}-\frac{n}{p}, \frac{n}{q}-\frac{n-1}{2}\big\}.
\end{equation}
To express $\gamma(p, q)$ more clearly, with a little calculation, we divide the region $I$ of $\mathbb{R}^{2}$ into four parts:
\begin{align}
& \mathcal{U}=\left\{(x, y)\in I, x-y\geqslant\frac{2}{n+1}, x>\frac{n+1}{2n}, y<\frac{n-1}{2n}\right\},\\
& \mathcal{V}=\left\{(x, y)\in I, 0\leqslant x-y<\frac{2}{n+1}, \frac{n-1}{n+1}(1-x)\leqslant y\leqslant\frac{n+1}{n-1}(1-x)\right\},\\
& \mathcal{W}=\left\{(x, y)\in I, y<\frac{n-1}{n+1}(1-x), y\leqslant x<\frac{n+1}{2n}\right\},
\end{align}
and the dual $\mathcal{W}'$ of $\mathcal{W}$.
Set $C=(\frac{1}{2}, \frac{1}{2})$, $B=(\frac{n-1}{2n}, \frac{n-1}{2n})$ and $B'= (\frac{n+1}{2n}, \frac{n+1}{2n})$. For a number of points $L_{1}, L_{2}, \cdots, L_{k}$, let $[L_{1}, L_{2}, \cdots, L_{k}]$ denote the convex hull of them. For two points $M$ and $N$, we also use $[M, N)$ to represent the half open line segment connecting $M$ to $N$, $N$ excluded. $(M, N]$ and $(M, N)$ are defined similarly.
With these points and notations, we define
\begin{align}
& \mathcal{R}_{1}=\mathcal{U}\cap\mathcal{R}_{0},\\
& \mathcal{R}_2=\mathcal{V}\backslash([B, C)\cup[B', C)),\\
& \mathcal{R}_3=\mathcal{W}\cap\mathcal{R}_{0}.
\end{align}
We can then express $\gamma_{p, q}$ case by case:
\begin{equation}
	\gamma_{p, q}=\begin{cases}
	0 & \text{if}\quad (\frac{1}{p}, \frac{1}{q})\in\mathcal{R}_{1},\\
	1-\frac{n+1}{2}(\frac{1}{p}-\frac{1}{q}) & \text{if}\quad (\frac{1}{p}, \frac{1}{q})\in\mathcal{R}_{2},\\
	\frac{n+1}{2}-\frac{n}{p} & \text{if}\quad (\frac{1}{p}, \frac{1}{q})\in\mathcal{R}_{3},\\
	\frac{n}{q}-\frac{n-1}{2} & \text{if}\quad (\frac{1}{p}, \frac{1}{q})\in\mathcal{R}_{3}'.
	\end{cases}
\end{equation}

Finally, for $(\frac{1}{p}, \frac{1}{q})\in\mathcal{R}_{1}\cup\mathcal{R}_{2}\cup\mathcal{R}_{3}\cup\mathcal{R}_{3}'$ and $z\in\mathbb{C}\backslash[0, \infty)$, we set \begin{equation} \label{Kappa}
\kappa_{p, q}(z)=|z|^{\frac{n}{2}(\frac{1}{p}-\frac{1}{q})-1+\gamma_{p, q}}\mathrm{dist} (z, [0, \infty))^{-\gamma_{p, q}}.
\end{equation}
Kwon and Lee conjectured in \cite{KL} that the following resolvent estimates hold whenever $(\frac{1}{p}, \frac{1}{q})\in\mathcal{R}_{1}\cup\mathcal{R}_{2}\cup\mathcal{R}_{3}\cup\mathcal{R}_{3}'$:
\begin{equation} \label{Res}
\|u\|_{L^{q}(\mathbb{R}^{n})}\leqslant C\kappa_{p, q}(z)\|(-\Delta-z)u\|_{L^{p}(\mathbb{R}^{n})}, \quad z\in\mathbb{C}\backslash [0, \infty),
\end{equation}
where the constant $C$ is independent of the complex number $z$. (They also conjectured that \eqref{Res} does not hold for $(\frac{1}{p}, \frac{1}{q})\in [B, C)\cup [B', C)$.) As we pointed out above, they proved this conjecture for almost all pairs $(\frac{1}{p}, \frac{1}{q})$ in $\mathcal{R}_{1}\cup\mathcal{R}_{2}\cup\mathcal{R}_{3}\cup\mathcal{R}_{3}'$, leaving only two small triangles of higher dimensional cases unsolved. More precisely, let $P=(\frac{1}{p_{1}}, \frac{1}{p_{1}})$ where
\begin{equation} \label{p1} \frac{1}{p_{1}}=\begin{cases}
\frac{3(n-1)}{2(3n+1)} & \text{if}\quad n\quad \text{is odd}\\
\frac{3n-2}{2(3n+2)} & \text{if}\quad n\quad \text{is even}
\end{cases},\end{equation}
and $Q=(\frac{1}{q_{1}}, \frac{1}{q_{2}})$ where \begin{equation} \label{q1}
\big(\frac{1}{q_{1}}, \frac{1}{q_{2}}\big)=\begin{cases}
(\frac{(n+5)(n-1)}{2(n^{2}+4n-1)}, \frac{(n-1)(n+3)}{2(n^{2}+4n-1)}) & \text{if}\quad n\quad \text{is odd}\\
(\frac{(n^{2}+3n-6)}{2(n^{2}+3n-2)}, \frac{(n-1)(n+2)}{2(n^{2}+3n-2)}) & \text{if}\quad n\quad \text{is even}
\end{cases}.\end{equation}
The complicated $(\frac{1}{p_{1}}, \frac{1}{p_{1}})$ and $(\frac{1}{q_{1}}, \frac{1}{q_{2}})$ arise from application of the oscillatory integral theorem of Guth, Hickman and Ilopoulou \cite{GHI} and that of the multilinear estimates of Tao \cite{Tao}. Kwon and Lee showed that \eqref{Res} holds for pairs of exponents $(\frac{1}{p}, \frac{1}{q})$ in the region $\mathcal{R}_{1}\cup\tilde{\mathcal{R}}_{2}\cup\tilde{\mathcal{R}}_{3}\cup\tilde{\mathcal{R}}_{3}'$, where
\[\tilde{\mathcal{R}}_{2}=\mathcal{R}_{2}\backslash([B, Q, C]\cup [P', Q', C])\cap\{C\},\]
and
\[\tilde{\mathcal{R}}_{3}=\mathcal{R}_{3}\backslash [B, P, Q].\]
Note that when $n=2$, $\tilde{\mathcal{R}}_{2}=\mathcal{R}_{2}$ and $\tilde{\mathcal{R}}_{3}=\mathcal{R}_{3}$, so the two dimensional conjecture is no longer a conjecture now. A picture of the regions $\mathcal{R}_{1}$, $\tilde{\mathcal{R}}_{2}$, $\tilde{\mathcal{R}}_{3}$ and $\tilde{\mathcal{R}}_{3}'$ in the coordinate plane is provided below (Figure \ref{pqFig}).

\maketitle
\section{Main Theorems and their Application}

We are now in a good position to state our main theorems. We need to restrict ourselves to those complex numbers $z$ so that the resolvent estimates in Kwon-Lee \cite{KL} are uniform, i.e., the bounds do not depend on $z$. To this end, we define the region
\begin{equation}
\mathcal{Z}_{p, q}=\{z\in\mathbb{C}\backslash [0, \infty): \kappa_{p, q}(z)\leqslant 1\}.
\end{equation}
(We may replace the $1$ by any universal constant $M$.) When $z\in\mathcal{Z}_{p, q}$, it follows easily from \eqref{Res} that for $(\frac{1}{p}, \frac{1}{q})\in\mathcal{R}_{1}\cup\tilde{\mathcal{R}_{2}}\cup\tilde{\mathcal{R}_{3}}\cup\tilde{\mathcal{R}_{3}}'$,
\begin{equation} \label{URes}
\|u\|_{L^{q}(\mathbb{R}^{n})}\leqslant C\|(-\Delta-z)u\|_{L^{p}(\mathbb{R}^{n})},
\end{equation}
where the constant $C$ is independent of the complex number $z\in\mathcal{Z}_{p, q}$. We provide pictures of $\mathcal{Z}_{p, q}$ in the complex plane when $(\frac{1}{p}, \frac{1}{q})$ belongs to different regions; see Figure \ref{zFig} below. In that figure, \begin{align*}
& \tilde{\mathcal{R}}_{3, \pm}:=\Big\{(x, y)\in\tilde{\mathcal{R}}_{3}: \pm\Big (x+y-\frac{n-1}{n}\Big )>0\Big\},\\
& \tilde{\mathcal{R}}_{3, 0}:=\Big\{(x, y)\in\tilde{\mathcal{R}}_{3}: \Big (x+y-\frac{n-1}{n}\Big)=0\Big\}.
\end{align*}

\begin{figure}
	\centering
	\includegraphics[height=7.5cm]{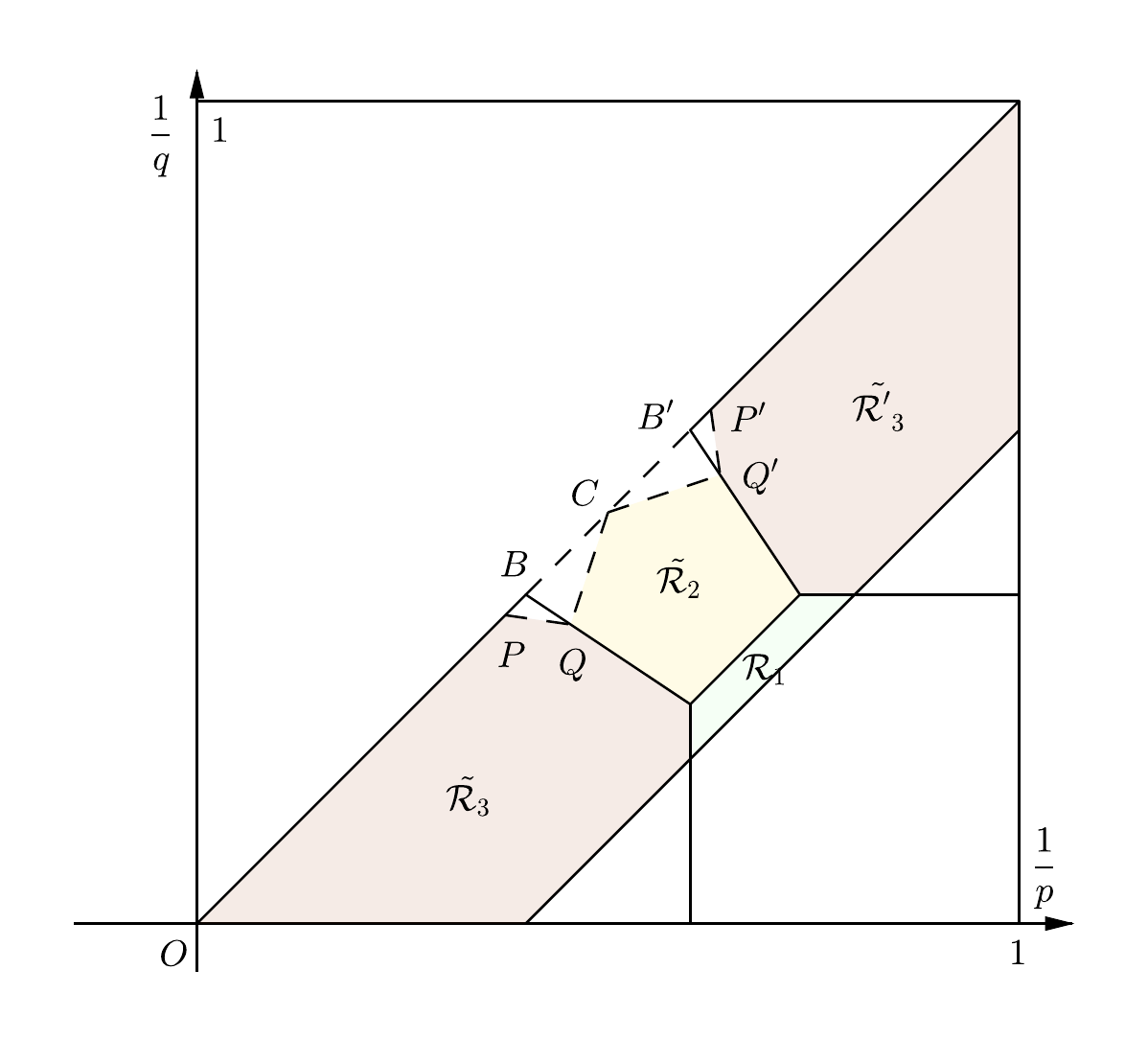}
	\caption{Region for the exponent pair $(\frac{1}{p}, \frac{1}{q})$ in resolvent estimates}
	\label{pqFig}
\end{figure}

\begin{figure}[htbp]
	\centering
	\subfloat[$(\frac{1}{p}, \frac{1}{q})\in\mathcal{R}_{1}$]
	{
		\begin{minipage}[t]{0.5\textwidth}
			\centering
			\includegraphics[width=0.5\textwidth]{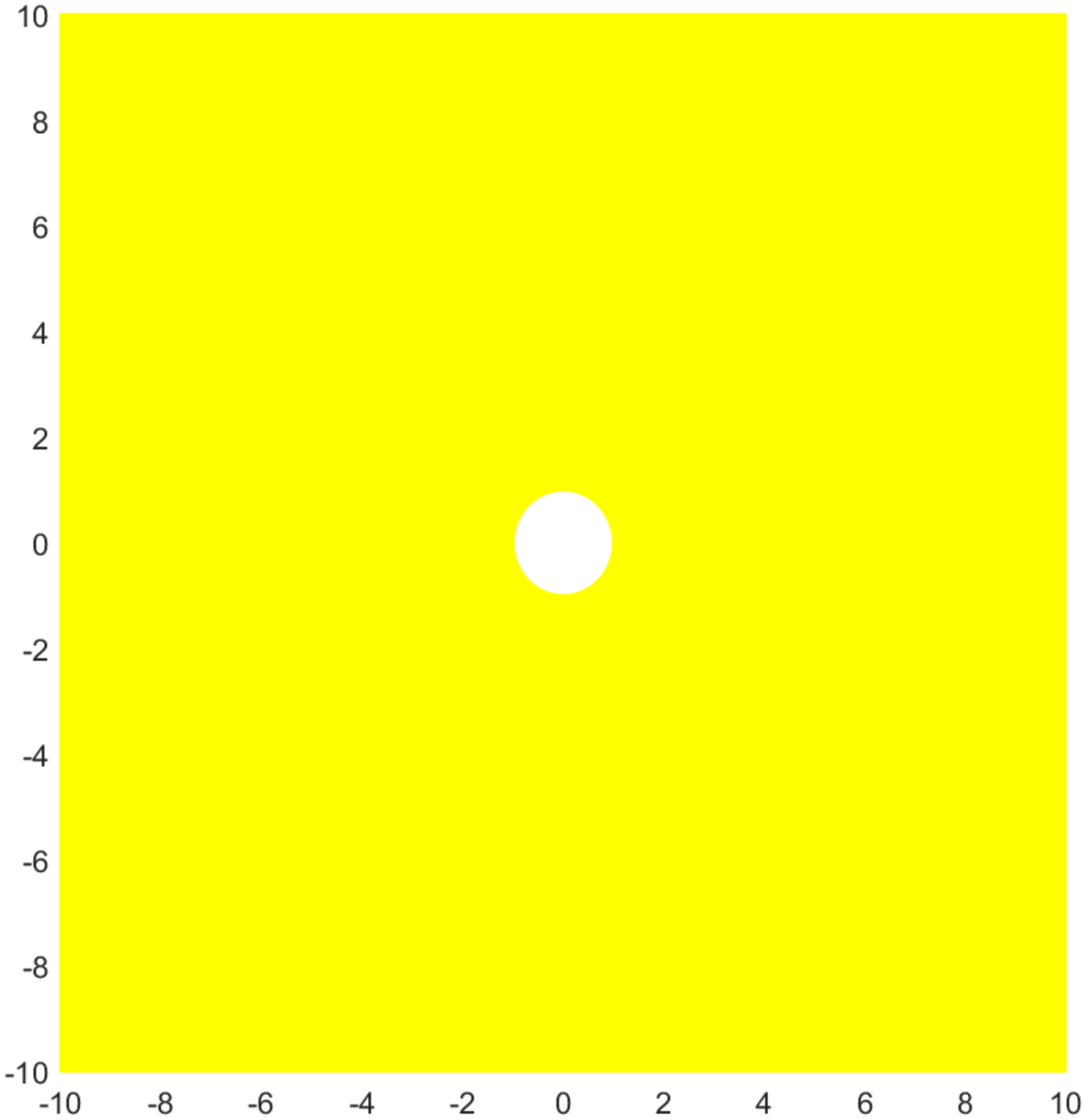}
		\end{minipage}
	}
	\subfloat[$(\frac{1}{p}, \frac{1}{q})\in\tilde{\mathcal{R}}_{2}\backslash \{C\}\cup\tilde{\mathcal{R}}_{3, +}\cup\tilde{\mathcal{R}}'_{3, +}$]
	{
		\begin{minipage}[t]{0.5\textwidth}
			\centering
			\includegraphics[width=0.5\textwidth]{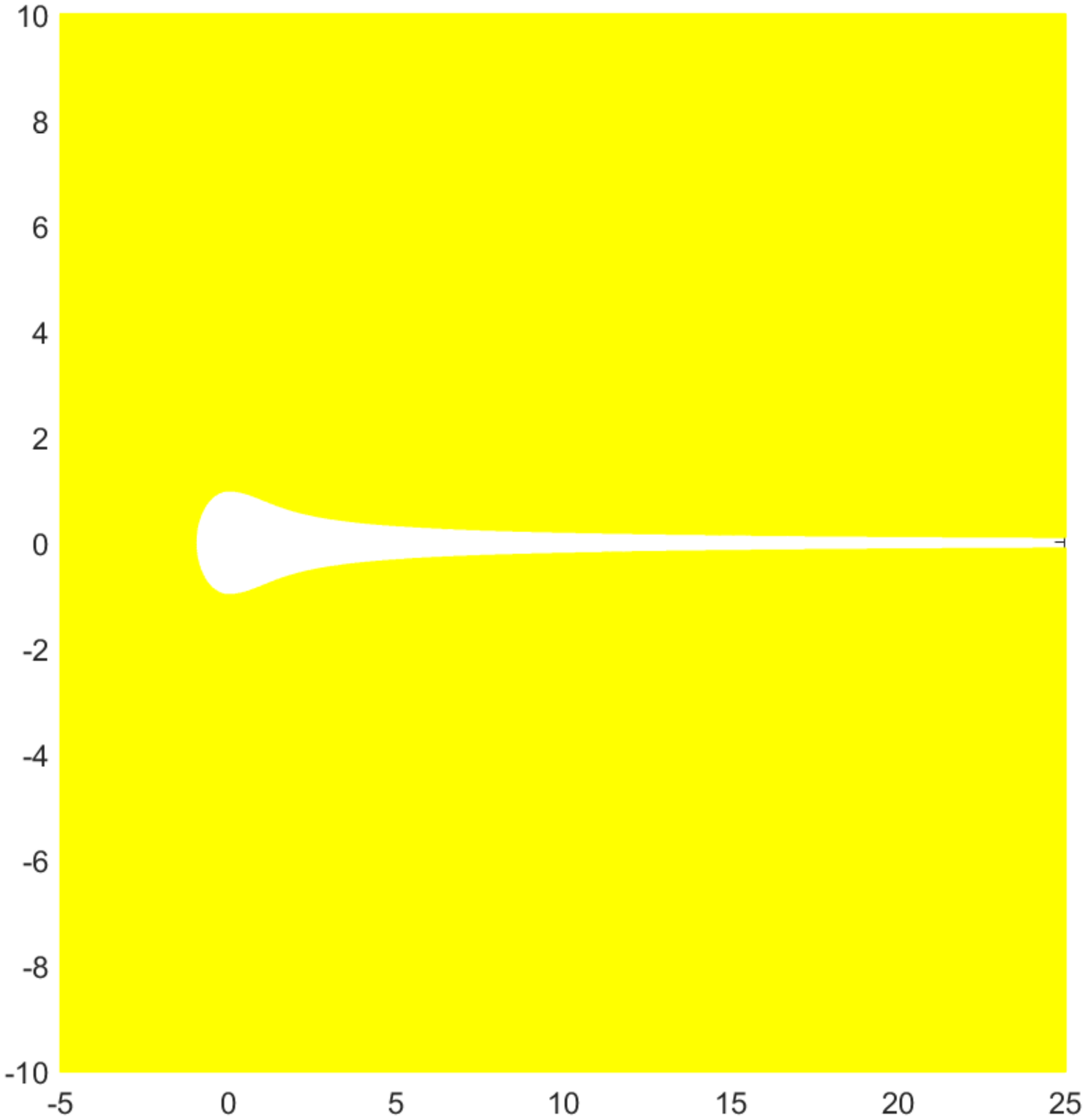}
		\end{minipage}
	}
	
	\subfloat[$(\frac{1}{p}, \frac{1}{q})\in\tilde{\mathcal{R}}_{3, 0}\cup\tilde{\mathcal{R}}'_{3, 0}\cup\{C\}$]
	{
		\begin{minipage}[t]{0.5\textwidth}
			\centering
			\includegraphics[width=0.5\textwidth]{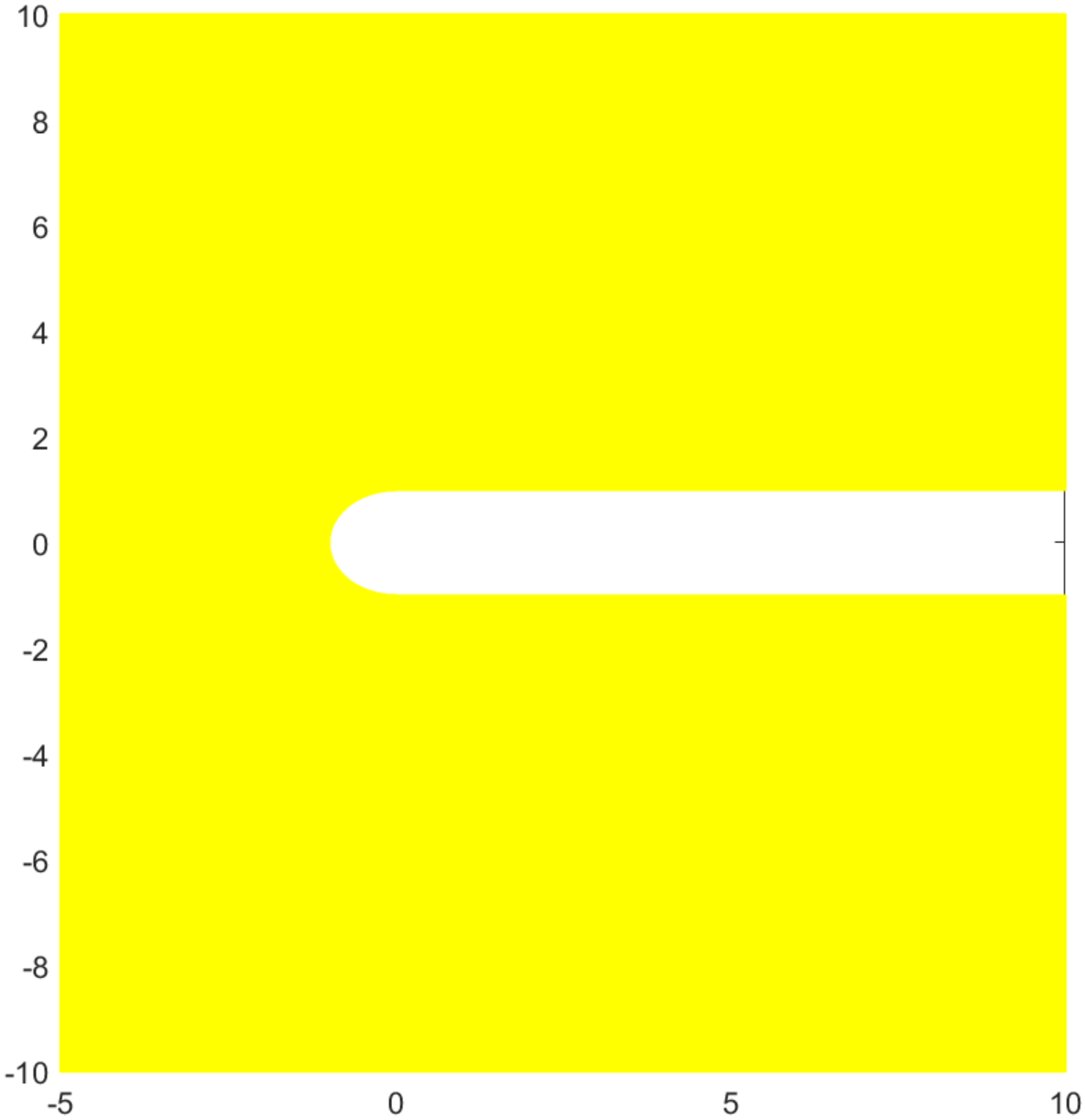}
		\end{minipage}
	}
	\subfloat[$(\frac{1}{p}, \frac{1}{q})\in\tilde{\mathcal{R}}_{3, -}\cup\tilde{\mathcal{R}}'_{3, -}$]
	{	
		\begin{minipage}[t]{0.5\textwidth}
			\centering	
			\includegraphics[width=0.5\textwidth]{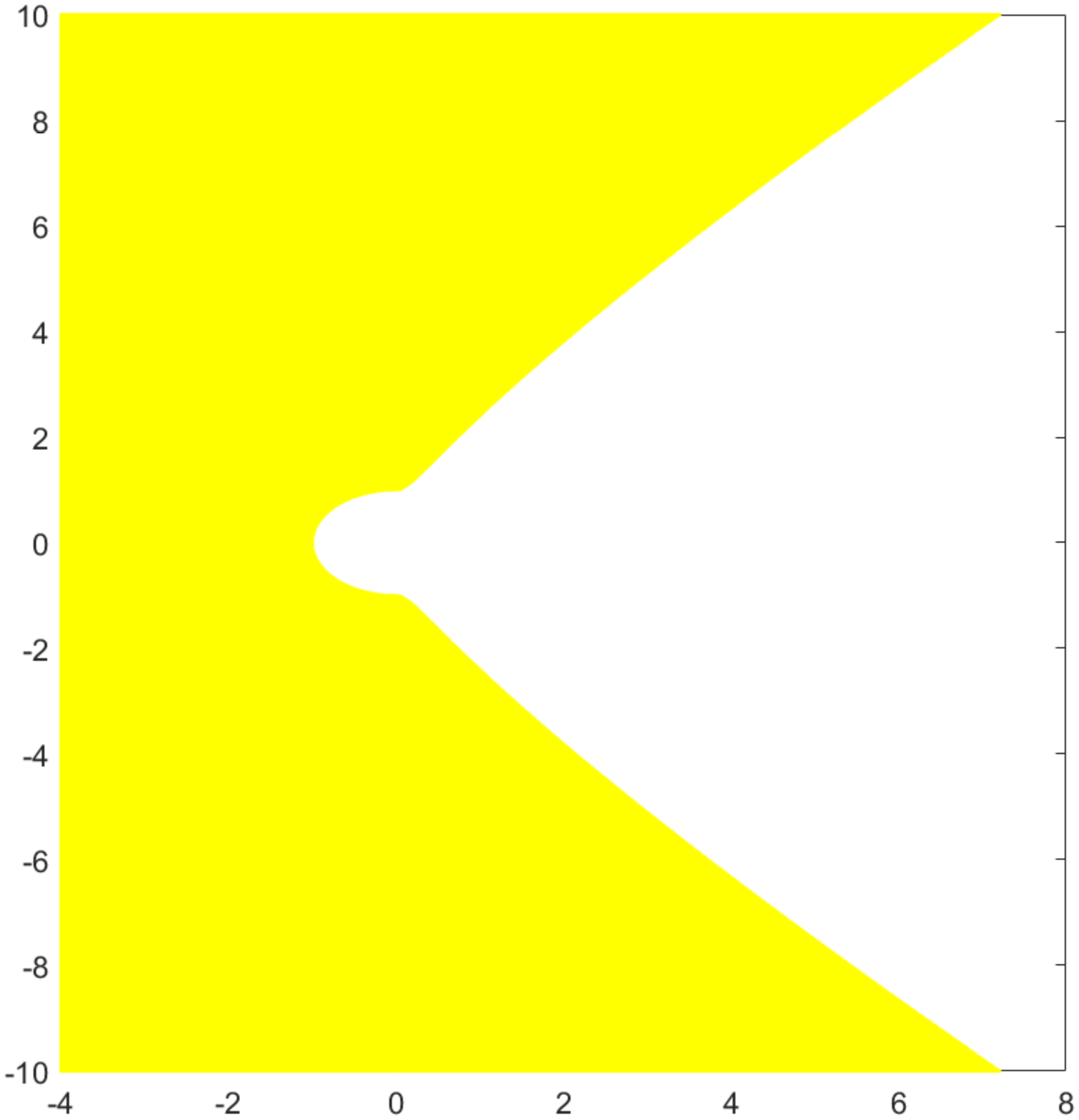}
		\end{minipage}
	}
	\caption{Regions for the complex number $z$ to get uniform estimates in different cases}
	\label{zFig}
\end{figure}

Our theorem is:
\begin{theorem} \label{Main1}
	Suppose that $(\frac{1}{p}, \frac{1}{q})\in\mathcal{R}_{1}\cup\tilde{\mathcal{R}_{2}}\cup\tilde{\mathcal{R}_{3}}\cup\tilde{\mathcal{R}_{3}}'$, $p\neq q$, and $z\in\mathcal{Z}_{p, q}$. Denote $\frac{1}{\sigma}=\frac{1}{p}-\frac{1}{q}$. Suppose also that $V(x)\in (L^{\sigma}(\mathbb{R}^{n})+L^{\infty}(\mathbb{R}^{n}))\cap\mathcal{K}(\mathbb{R}^{n})$. Then there exists a constant $C$, depending only on $n$, $p$, $q$ and $V$, such that the following inequality holds:
	\begin{equation} \label{ResV1}
	\|u\|_{L^{q}(\mathbb{R}^{n})}\leqslant C\kappa_{p, q}(z)(\|(-\Delta+V-z)u\|_{L^{p}(\mathbb{R}^{n})}+|z|^{\frac{1}{2}}\|u\|_{L^{p}(\mathbb{R}^{n})}), \quad u\in L^{q}(\mathbb{R}^{n}).
	\end{equation}
\end{theorem}

\noindent If we discard the bounded $\kappa_{p, q}(z)$ in Theorem \ref{Main1} (for instance, when considering those $z$ such that $\kappa_{p, q}(z)$ is away from $0$), we obtain immediately the following
\begin{theorem} \label{Main}
	Suppose that $(\frac{1}{p}, \frac{1}{q})\in\mathcal{R}_{1}\cup\tilde{\mathcal{R}_{2}}\cup\tilde{\mathcal{R}_{3}}\cup\tilde{\mathcal{R}_{3}}'$, $p\neq q$, and $z\in\mathcal{Z}_{p, q}$. Denote $\frac{1}{\sigma}=\frac{1}{p}-\frac{1}{q}$. Suppose also that $V(x)\in (L^{\sigma}(\mathbb{R}^{n})+L^{\infty}(\mathbb{R}^{n}))\cap\mathcal{K}(\mathbb{R}^{n})$. Then there exists a constant $C$, depending only on $n$, $p$, $q$ and $V$, such that the following inequality holds:
	\begin{equation} \label{ResV}
	\|u\|_{L^{q}(\mathbb{R}^{n})}\leqslant C(\|(-\Delta+V-z)u\|_{L^{p}(\mathbb{R}^{n})}+|z|^{\frac{1}{2}}\|u\|_{L^{p}(\mathbb{R}^{n})}), \quad u\in L^{q}(\mathbb{R}^{n}).
	\end{equation}
\end{theorem}

If $p=q$, i.e. $\sigma=\infty$, we need additionally assume that $\|V\|_{L^{\infty}(\mathbb{R}^{n})}$ is small enough. Then we would get the better inequality
\[\|u\|_{L^{q}(\mathbb{R}^{n})}\leqslant C(\|(-\Delta+V-z)u\|_{L^{p}(\mathbb{R}^{n})}, \quad u\in L^{q}(\mathbb{R}^{n}).\]
See the proof of Corollary 4.

Before moving on to the proof of the main theorems, we take a look at the application of them in the spectral theory for ``$-\Delta+V$''. First, let $z=(\lambda+i)^{2}$ with $\lambda\geqslant 1$, and assume $(\lambda+i)^{2}\in\mathcal{Z}_{p, q}$ (e.g. this is true for any $\lambda\geqslant 1$ when $(\frac{1}{p}, \frac{1}{q})\notin\tilde{\mathcal{R}}_{3, -}\cup\tilde{\mathcal{R}}'_{3, -}$). Substituting for $z$ in \eqref{ResV1} and recalling the expression \eqref{Kappa} of $\kappa_{p, q}(z)$ yield the following extension of Theorem 9.1 in Blair-Sire-Sogge \cite{BSS}:

\begin{corollary} \label{Cor1}
	Suppose the exponents $p$, $q$ and the potential $V$ are as in Theorems \ref{Main1} and \ref{Main}. Suppose $\lambda\geqslant 1$ is a real number such that $(\lambda+i)^{2}\in\mathcal{Z}_{p, q}$. Then we have
	\begin{multline} \label{Extension}
	\|u\|_{L^{q}(\mathbb{R}^{n})}\leqslant C(\lambda^{n(\frac{1}{p}-\frac{1}{q})-2+\gamma_{p, q}}\|(-\Delta+V-(\lambda+i)^{2})u\|_{L^{p}(\mathbb{R}^{n})}\\
	+\lambda^{n(\frac{1}{p}-\frac{1}{q})-1+\gamma_{p, q}}\|u\|_{L^{p}(\mathbb{R}^{n})}), \quad u\in L^{q}(\mathbb{R}^{n}).
	\end{multline}
	with the constant $C$ independent of $\lambda$ (it only depends on $n$, $p$, $q$ and $V$).
\end{corollary}

\noindent Indeed, one checks by an easy computation that when $p=2$, $2<q\leqslant\infty$ in the $n=2$ and $n=3$ cases and $2<q<\frac{2n}{n-4}$ in the $n>3$ cases, the power $n(\frac{1}{p}-\frac{1}{q})-2+\gamma_{p, q}$ on $\lambda$ in \eqref{Extension} is exactly the $\sigma(q)-1$ in Theorem 9.1 of Blair-Sire-Sogge \cite{BSS}. 

Let $\chi_{\lambda}^{V}$ denote the spectral projection operator associated with $-\Delta+V$ onto the interval $[\lambda^{2}, (\lambda+1)^{2})$. Substituting $u=\chi_{\lambda}^{V}f$ in \eqref{Extension} where $f\in\mathscr{S}(\mathbb{R}^{n})$ and combining the two terms on the right of the resulting inequality give:

\begin{corollary} \label{Cor2}
$p$, $q$, $V$ and $\lambda$ as in Corollary \ref{Cor1}. We have
\begin{equation}
\|\chi_{\lambda}^{V}f\|_{L^{q}(\mathbb{R}^{n})}\leqslant C\lambda^{n(\frac{1}{p}-\frac{1}{q})-1+\gamma_{p, q}}\|\chi_{\lambda}^{V}f\|_{L^{p}(\mathbb{R}^{n})}, \quad f\in\mathscr{S}(\mathbb{R}^{n}).
\end{equation}
\end{corollary}
\noindent This extends estimate (9.3) in Blair-Sire-Sogge \cite{BSS}.

We can actually improve the above results in some sense. At the end of this article, we will prove the following corollary that extends Theorem 9.2 in \cite{BSS}:
\begin{corollary} \label{Cor3}
	$p$, $q$ and $V$ as in Theorems \ref{Main1} and \ref{Main}. Suppose $\lambda>0$ and $0<\mu<\frac{\lambda}{2}$ are real numbers such that $(\lambda+i\mu)^{2}\in\mathcal{Z}_{p, q}$. Then if $\|V\|_{L^{\sigma}(\mathbb{R}^{n})}$ is small enough, we have
	\begin{equation} \label{Extension2}
	\|u\|_{L^{q}(\mathbb{R}^{n})}\leqslant C\lambda^{n(\frac{1}{p}-\frac{1}{q})-2+\gamma_{p, q}}\mu^{-\gamma_{p, q}}\|(-\Delta+V-(\lambda+i\mu)^{2})u\|_{L^{p}(\mathbb{R}^{n})}, \quad u\in L^{q}(\mathbb{R}^{n}).
	\end{equation}
\end{corollary}
\noindent Let $\chi_{[\lambda, \lambda+\mu)}^{V}$ denote the spectral projection operator associated with $-\Delta+V$ corresponding to the interval $[\lambda^{2}, (\lambda+\mu)^{2})$. Plugging $\chi_{[\lambda, \lambda+\mu)}^{V}f$ in \eqref{Extension2} gives the following extension of Corollary 9.4 in \cite{BSS}:

\begin{corollary} \label{Cor4}
	$p$, $q$, $V$, $\lambda$ and $\mu$ as in Corollary \ref{Cor3}, then we have
	\begin{equation} \label{spectral}
	\|\chi_{[\lambda, \lambda+\mu)}^{V}f\|_{L^{q}(\mathbb{R}^{n})}\leqslant C\lambda^{n(\frac{1}{p}-\frac{1}{q})-1+\gamma_{p, q}}\mu^{1-\gamma_{p, q}}\|\chi_{[\lambda, \lambda+\mu)}^{V}f\|_{L^{p}(\mathbb{R}^{n})}, \quad f\in\mathscr{S}(\mathbb{R}^{n}).
	\end{equation} 
	In addition, if $\mu$ is bounded away from $0$, i.e. $\mu\geqslant\delta_{0}$ for some $\delta_{0}>0$, it follows that
	\begin{equation}
	\|\chi_{[0, 2\mu)}^{V}f\|_{L^{q}(\mathbb{R}^{n})}\leqslant C\mu^{n(\frac{1}{p}-\frac{1}{q})}\|\chi_{[0, 2\mu)}^{V}f\|_{L^{p}(\mathbb{R}^{n})}, \quad f\in\mathscr{S}(\mathbb{R}^{n}).
	\end{equation}
	The constants $C$ in the conclusions are independent from $\lambda$ and $\mu$.
\end{corollary}
\noindent The estimates in the above corollaries are closely related to the Stein-Tomas restriction theorem and the Strichartz estimates.

\maketitle
\section{Proof of the Main Theorems}
It suffices to prove Theorem \ref{Main1}. To prove Theorem \ref{Main1}, we basically follow the ideas of Blair-Sire-Sogge \cite{BSS}. However, their proof uses the expression and estimates of the kernel of the resolvent operator in Kenig-Ruiz-Sogge's 1986 paper \cite{KRS}, but the proof of Kwon-Lee \cite{KL} is a novel one, applying recent techniques such as the oscillatory integral estimates of Guth-Hickman-Iliopoulou \cite{GHI} and the bilinear estimates of Tao \cite{Tao}. For this reason, many lines need to be carefully justified to ensure that Blair-Sire-Sogge \cite{BSS}'s proof combines well with this new method.

We deal with cases when $q\neq\infty$ first. In this situation, we will see that we do not even need the condition that $V(x)$ be in the Kato class.

Let \begin{equation}
F_{z}(x,y)=\frac{1}{(2\pi)^{n}}\int_{\mathbb{R}^{n}} \frac{e^{i(x-y)\cdot\xi}}{|\xi|^{2}-z}d\xi
\end{equation}
denote the kernel for the resolvent operator $(-\Delta-z)^{-1}$. The idea is to break $\mathbb{R}^{n}$ into small cubes on each of which the $L^{\sigma}$ norm of $V$ is small enough and then sum them up. To realize this, we introduce an $\eta\in C_{0}^{\infty}(\mathbb{R}^{n})$ that is supported in $[-1, 1]$ and equals $1$ in $[-\frac{1}{2}, \frac{1}{2}]$, and set $\eta_{\delta}(x, y)=\eta(\frac{|x-y|}{\delta})$ where the $\delta$ is to be specified later. Let
\begin{equation}
H_{z}(x, y)=\eta_{\delta}(x, y)F_{z}(x, y).
\end{equation}
We have by an easy calculation \begin{equation} \label{1}
(-\Delta-z)H_{z}(x, y)=\delta_{y}(x)+[\eta_{\delta}(x, y), \Delta]F_{z}(x, y).
\end{equation}
Taking adjoints in the above equality, we get \begin{equation} \label{id}
I=H_{z}\circ (-\Delta-z)+R_{z},
\end{equation}
where $H_{z}$ is the operator whose kernel is $H_{z}(x, y$) and $R_{z}$ corresponds to $-[\eta_{\delta}(x, y), \Delta]F_{z}(x, y)$.
To fulfill the ideas in Blair-Sire-Sogge \cite{BSS}, we need estimates of $H_{z}$ and $R_{z}$, and this is where we must be careful.

\begin{proposition} \label{Prop1}
	Let $(\frac{1}{p}, \frac{1}{q})$ be as in Theorem \ref{Main1} and $z\in\mathbb{C}\backslash[0, \infty)$. Then $H_{z}$ bears the same estimates as the resolvent operator $F_{z}$:
	\begin{equation}
	\|H_{z}f\|_{L^{q}(\mathbb{R}^{n})}\leqslant C\kappa_{p, q}(z)\|f\|_{L^{p}(\mathbb{R}^{n})}, \quad u\in C_{0}^{\infty}(\mathbb{R}^{n}),
	\end{equation}
	with the constant $C$ depending only on $n$, $p$ and $q$ but not on $z$ and $\delta$.
 \end{proposition}

\begin{proof}
	By dilation, we may assume that $z\in\mathbb{S}^{n}\backslash\{1\}$:
	\begin{gather*}
	\Big\|\mathscr{F}^{-1}\Big\{\frac{1}{|\xi|^{2}-z/|z|}\cdot\hat{f}(\xi)\Big\}\Big\|_{L^{q}(\mathbb{R}^{n})}\leqslant A\|f\|_{L^{p}(\mathbb{R}^{n})}\\
	\Updownarrow\\
	\Big\|\mathscr{F}^{-1}\Big\{\frac{1}{|\xi|^{2}-z}\cdot\hat{f}(\xi)\Big\}\Big\|_{L^{q}(\mathbb{R}^{n})}\leqslant A|z|^{\frac{n}{2}(\frac{1}{p}-\frac{1}{q})-1}\|f\|_{L^{p}(\mathbb{R}^{n})}.
	\end{gather*}
	Under this assumption, we want to show that \begin{equation}
	\|H_{z}f\|_{L^{q}(\mathbb{R}^{n})}\leqslant C\mathrm{dist} (z, [0, \infty))^{-\gamma_{p, q}}\|f\|_{L^{p}(\mathbb{R}^{n})}, \quad u\in C_{0}^{\infty}(\mathbb{R}^{n}).
	\end{equation}
	
	Denote $\frac{1}{|\xi|^{2}-z}=m(\xi)$ for convenience. We decompose $m(\xi)$ into a part near the unit sphere (the one that makes the operator singular) and a part away from it. To do this, fix a small number $\epsilon_{0}>0$ and choose a $C_{0}^{\infty}(\mathbb{R}^{n})$ function $\rho(\xi)$ that equals $1$ when $1-\epsilon_{0}\leqslant|\xi|\leqslant 1+\epsilon_{0}$ and equals $0$ when $|\xi|\leqslant 1-2\epsilon_{0}$ or $|\xi|\geqslant 1+2\epsilon_{0}$. Let
	\begin{align*}
	& m_{1}(\xi)=(1-\rho(\xi))\mathbbm{1}_{B_{1}(O)}m(\xi),\\ & m_{2}(\xi)=(1-\rho(\xi))\mathbbm{1}_{B_{1}(O)^{c}}m(\xi), \quad \mathrm{and}\\ & m_{3}(\xi)=\rho(\xi)m(\xi),
	\end{align*}
	where $B_{1}(O)$ stands for the unit ball. $H_{z}$ is broken up into three parts accordingly.
	
	In the first place, since $m_{1}(\xi)$ is smooth and compactly supported, it follows that the multiplier operator defined by $m_{1}(\xi)$ is uniformly bounded (i.e. bounded by a universal constant $C$) from $L^{p}(\mathbb{R}^{n})$ to $L^{q}(\mathbb{R}^{n})$ for any $1\leqslant p\leqslant q\leqslant\infty$. Multiplying the kernel of this operator by the bounded function $\eta_{\delta}(x, y)$ does not have any essential effect on the boundedness of this operator, so we see that this part of $H_{z}$ is uniformly bounded from $L^{p}$ to $L^{q}$, a better result than needed. 
	
	Second, notice that $||\xi|^{2}-z|\geqslant 2\epsilon_{0}+\epsilon_{0}^{2}$ when $|\xi|\geqslant 1+\epsilon_{0}$, we have \begin{equation} \label{m2E}
	|\partial^{\alpha}_{\xi}m_{2}(\xi)|\leqslant C_{\alpha, \epsilon_{0}}|\xi|^{-|\alpha|-2}.
	\end{equation} Then we just do a standard work. Pick a Littlewood-Paley bump function $\beta(t)$ on $\mathbb{R}$ supported in $\{t\in\mathbb{R}: \frac{1}{2}\leqslant|t|\leqslant 1\}$ such that $\sum\limits_{j=-\infty}^{\infty}\beta(2^{-j}t)=1$ for all $t\neq 0$. Let $\beta_{0}(t)=\sum\limits_{j=-\infty}^{-1}\beta(2^{-j}t)$, so $\beta_{0}(t)$ is supported in $|t|\leqslant \frac{1}{2}$. Therefore, $\sum\limits_{j=1}^{\infty}\beta(2^{-j}|\xi|)=1$ on the support of $m_{2}(\xi)$. By an easy scaling argument together with the estimate \eqref{m2E}, we deduce for any $r\geqslant 1$, \begin{equation}
	\|\mathscr{F}^{-1}\{\beta(2^{-j}|\xi|)m_{2}(\xi)\}\|_{L^{r}(\mathbb{R}^{n})}\leqslant C2^{(n-\frac{n}{r}-2)j},
	\end{equation} where the constant $C$ is independent of $j$. From these, we know that the multiplier operator defined by $m_{2}(\xi)$ is uniformly bounded from $L^{p}(\mathbb{R}^{n})$ to $L^{q}(\mathbb{R}^{n})$ whenever $(\frac{1}{p}, \frac{1}{q})\in\mathcal{R}_{0}$. Indeed, when $\frac{1}{p}-\frac{1}{q}<\frac{2}{n}$, we apply Young's inequality and sum over $j$. When $\frac{1}{p}-\frac{1}{q}=\frac{2}{n}$, summing over $j$ does not work, but we can resort to Bourgain's interpolation \cite{Bourgain} to get restricted weak type estimates at these $(\frac{1}{p}, \frac{1}{q})$ first, and then apply real interpolation to get strong inequalities. See the appendix for Bonrgain's interpolation method. Once again, multiplying the kernel of the operator corresponding to $m_{2}(\xi)$ by $\eta_{\delta}(x, y)$ does not affect the boundedness of this operator, hence we see that this part of $H_{z}$ is also uniformly bounded from $L^{p}$ to $L^{q}$.
	
	Finally, we deal with $m_{3}(\xi)$, the singular part. Fix a small number $\theta_{0}>0$. If $z$ is in the set $\{z\in e^{i\theta}: \theta\in[\theta_{0}, 2\pi-\theta_{0}]\}$, then we easily have $|\partial^{\alpha}_{\xi}m_{3}(\xi)|\leqslant C_{\alpha, \epsilon_{0}, \theta_{0}}$, and by the same reasoning as for $m_{1}(\xi)$, we settle $m_{3}(\xi)$. If $z$ is close to $1$, note that after taking Fourier transform, the effect of $\eta_{\delta}(x, y)$ on $m_{3}(\xi)$ amounts to convolving $m_{3}(\xi)$ with a function whose $L^{1}$ norm is finite (of the form $\delta^{n} g(\delta x))$. By dilation, we may assume for convenience that $z=1+i\epsilon$ where $|\epsilon|$ is very small. In Kwon-Lee \cite{KL}, it was shown that for such $z$, \[\|m_{3}(D)f\|_{L^{q}(\mathbb{R}^{n})}\leqslant C\epsilon^{-\gamma_{p, q}}\|f\|_{L^{p}(\mathbb{R}^{n})}.\]
	We want to show here the same estimates for the part of $H_{z}$ corresponding to $m_{3}(\xi)$. To explain our proof, we outline the method in \cite{KL}.
	
	Write $\xi=(\xi', \xi_{n})$ where $\xi'\in\mathbb{R}^{n-1}$, $\xi_{n}\in\mathbb{R}$. After decomposing by a partition of unity, discarding the easy smooth part, rotating (so that we may assume that the multiplier is supported near $(0, \cdots, 0, -1)$), and making the change of variable $\xi_{n}\rightarrow\xi_{n}-1$ (now the multiplier is supported near the origin), we are able to express the essential part of $m_{3}(\xi)$ as
	\[\mathfrak{m}(\xi', \xi_{n})=\frac{1}{|\epsilon|}\phi\Big(\frac{r(\xi', \xi_{n})(\xi_{n}-\psi(\xi'))}{|\epsilon|}\Big)\chi_{0}(\xi', \xi_{n}),\]
	where \begin{align*}
	& \phi(t)=\frac{1}{2t\pm i},\\ 
	& \psi(\xi')=1-\sqrt{1-|\xi'|^{2}},\\
	& r(\xi', \xi_{n})=\frac{1}{2}(\xi_{n}+\psi(\xi')-2),
	\end{align*} and $\chi_{0}$ is a smooth function supported in a small neighborhood of the origin. By a further affine transformation, we can make $\psi\in\textbf{Ell}(N, e)$ and $r\in\textbf{Mul}(N, b)$. Here $\textbf{Ell}(N, e)$ and $\textbf{Mul}(N, b)$ are two function classes defined for a large number $N>0$, a small number $e>0$ and some number $b>0$ (to focus on our proof, we omit their definitions and the affine transformation which can be found in 2.2, 2.3 and Remark 7 in \cite{KL}). 
	
	Now we break $\mathfrak{m}(\xi', \xi_{n})$ into the following sum using the $\beta$ and $\beta_{0}$ as in the treatment of $m_{2}(\xi)$ above:
	\begin{multline}
	\mathfrak{m}(\xi', \xi_{n})=\frac{1}{|\epsilon|}\phi\Big(\frac{r(\xi',\xi_{n})(\xi_{n}-\psi(\xi'))}{|\epsilon|}\Big)\beta_{0}\Big(\frac{r(\xi', \xi_{n})(\xi_{n}-\psi(\xi'))}{|\epsilon|}\Big)\chi_{0}(\xi', \xi_{n})\\
	+\frac{1}{|\epsilon|}\sum\limits_{j=1}^{\mathrm{log}\frac{1}{\epsilon}}\phi\Big(\frac{r(\xi', \xi_{n})(\xi_{n}-\psi(\xi'))}{|\epsilon|}\Big)\beta\Big(\frac{r(\xi', \xi_{n})(\xi_{n}-\psi(\xi'))}{2^{j-1}|\epsilon|}\Big)\chi_{0}(\xi', \xi_{n}).
	\end{multline}
	For notational convenience, we denote for $\lambda\geqslant 1$, \begin{align*}
	& \mathfrak{m}_{0}(\xi', \xi_{n})=\phi\Big(\frac{r(\xi', \xi_{n})(\xi_{n}-\psi(\xi'))}{|\epsilon|}\Big)\beta_{0}\Big(\frac{r(\xi', \xi_{n})(\xi_{n}-\psi(\xi'))}{|\epsilon|}\Big)\chi_{0}(\xi', \xi_{n}),\\
	& \mathfrak{m}_{\lambda}(\xi', \xi_{n})=\phi\Big(\frac{r(\xi', \xi_{n})(\xi_{n}-\psi(\xi'))}{|\epsilon|}\Big)\beta\Big(\frac{r(\xi', \xi_{n})(\xi_{n}-\psi(\xi'))}{\lambda|\epsilon|}\Big)\chi_{0}(\xi', \xi_{n}).
	\end{align*}
	Then the question transforms to analyzing the operators $\mathfrak{m}_{0}(D)$ and $\mathfrak{m}_{\lambda}(D)$, where in accordance with convention, $D=\frac{1}{i}(\partial_{x_{1}}, \cdots, \partial_{x_{n}})$. The key estimates related to these two operators are:
	\begin{equation} \label{Prop2.4}
	\begin{aligned}
	\text{(i)}\ & \|\mathfrak{m}_{0}(D)f\|_{L^{q}(\mathbb{R}^{n})}\leqslant C\epsilon^{\frac{1-n}{2}+\frac{n}{p}}\|f\|_{L^{p}(\mathbb{R}^{n})},\\
	& \|\mathfrak{m}_{\lambda}(D)f\|_{L^{q}(\mathbb{R}^{n})}\leqslant C\lambda^{-1}(\epsilon\lambda)^{\frac{1-n}{2}+\frac{n}{p}}\|f\|_{L^{p}(\mathbb{R}^{n})},
	\end{aligned}
	\end{equation} for $p, q$ satisfying $\frac{1}{q}=\frac{n-1}{n+1}(1-\frac{1}{p})$ and $q_{2}<q\leqslant\frac{2(n+1)}{n-1}$. (This is Proposition 2.4 in \cite{KL}. Recall that $q_{2}$ is defined by \eqref{q1} in Section \ref{Section2}.) The constant $C$ in these inequalities is independent of $\epsilon$, $\lambda$, $\psi\in\textbf{Ell}(N, e)$, $r\in\textbf{Mul}(N, b)$ and $f$.
	\begin{equation} \label{Prop2.5}
	\begin{aligned}
	\text{(ii)}\ & \|\mathfrak{m}_{0}(D)f\|_{L^{p}(\mathbb{R}^{n})}\leqslant C\epsilon^{\frac{1-n}{2}+\frac{n}{p}}\|f\|_{L^{p}(\mathbb{R}^{n})},\\
	& \|\mathfrak{m}_{\lambda}(D)f\|_{L^{p}(\mathbb{R}^{n})}\leqslant C\lambda^{-1}(\epsilon\lambda)^{\frac{1-n}{2}+\frac{n}{p}}\|f\|_{L^{p}(\mathbb{R}^{n})},
	\end{aligned} 
	\end{equation} for $p_{1}<p\leqslant\infty$. (This is Proposition 2.5 in \cite{KL}. Recall that $p_{1}$ is defined by \eqref{p1} in Section \ref{Section2}.) Again, the constant $C$ here is independent of $\epsilon$, $\lambda$, $\psi\in\textbf{Ell}(N, e)$, $r\in\textbf{Mul}(N, b)$ and $f$.
	
	The above key estimates \eqref{Prop2.4}, \eqref{Prop2.5} are proved using the recent oscillatory integral theorem of Guth-Hickman-Iliopoulou \cite{GHI} and the bilinear estimate of Tao \cite{Tao}. The rest of Kwon-Lee's proof is then mere interpolation and summation in $j$. 
	
	Let \begin{align*}
	& A_{0}(\xi)=\phi\Big(\frac{\tilde{r}(\xi', \xi_{n})\xi_{n}}{|\epsilon|}\Big)\beta_{0}\Big(\frac{\tilde{r}(\xi', \xi_{n})\xi_{n}}{|\epsilon|}\Big)\chi(\xi'),\\
	& A_{\lambda}(\xi)=\phi\Big(\frac{\tilde{r}(\xi',\xi_{n})\xi_{n}}{|\epsilon|}\Big)\beta\Big(\frac{\tilde{r}(\xi', \xi_{n})\xi_{n}}{\lambda|\epsilon|}\Big)\chi(\xi')
	\end{align*} where $\tilde{r}(\xi', \xi_{n})=r(\xi', \xi_{n}+\psi(\xi'))$. The essential ingredients that the proofs of the two key estimates \eqref{Prop2.4}, \eqref{Prop2.5} rely on are the following observations (Lemma 2.6 in \cite{KL}):
	
	\noindent (i) For every $n-1$ dimensional multi-index $\alpha=(\alpha_{1}, \cdots, \alpha_{n-1})$ with $|\alpha|\leqslant N$, we have \begin{equation} \label{A1}
	\begin{aligned} 
	& |\partial_{\xi'}^{\alpha}A_{0}(\xi)|\leqslant C_{\alpha},\\
	& |\partial_{\xi'}^{\alpha}A_{\lambda}(\xi)|\leqslant C_{\alpha}\lambda^{-1}
	\end{aligned} \end{equation}
	
	\noindent (ii) For every $n-1$ dimensional multi-index $\alpha$, $n$ dimensional multi-index $\zeta=(\zeta', \zeta_{n})$ with $|\alpha|+|\zeta|\leqslant N$, and every natural number $l$, we have
	\begin{equation} \label{A2}
	\begin{aligned}
	& |\partial_{\xi'}^{\zeta'}\partial_{\xi_{n}}^{\zeta_{n}}((\xi_{n})^{l}\partial_{\xi'}^{\alpha}A_{0}(\xi))|\leqslant C_{\alpha, \zeta}(\epsilon)^{-\zeta_{n}+l},\\
	& |\partial_{\xi'}^{\zeta'}\partial_{\xi_{n}}^{\zeta_{n}}((\xi_{n})^{l}\partial_{\xi'}^{\alpha}A_{\lambda}(\xi))|\leqslant C_{\alpha, \zeta}\lambda^{-1}(\epsilon\lambda)^{-\zeta_{n}+l}.
	\end{aligned} \end{equation} In these observations, the constants $C_{\alpha}$, $C_{\alpha, \zeta}$ are independent of $\epsilon$, $\lambda$, $\psi\in\textbf{Ell}(N, e)$ and $r\in\textbf{Mul}(N, b)$.
	
	Notice the right sides of the above two estimates \eqref{A1}, \eqref{A2} do not depend on $\xi$, so convolving with a function with finite $L^{1}$ norm does not affect these size estimates. And obviously, the affine transformations at the begining has no impact on these estimates--we simply make the same affine transformations to the convolution integral. Therefore, \eqref{A1}, \eqref{A2} still hold when the kernel of the operator $m_{3}(D)$ is multiplied by $\eta_{\delta}(x, y)$. It then follows that the key estimates \eqref{Prop2.4} and \eqref{Prop2.5} are also valid for our operator. By Kwon-Lee's proof, we come to the conclusion that we have the same estimates for this part of $H_{z}$ as those for the multiplier operator defined by $m_{3}(\xi)$. The proposition is proved thereby.
\end{proof}

\begin{proposition} \label{Prop2}
	$(\frac{1}{p}, \frac{1}{q})$ are still as in Theorem \ref{Main1} and $z\in\mathbb{C}\backslash[0, \infty)$. For $R_{z}$ we have
	\begin{equation}
	\|R_{z}f\|_{L^{q}(\mathbb{R}^{n})}\leqslant C_{\delta}\kappa_{p, q}(z)(1+|z|^{\frac{1}{2}})\|f\|_{L^{p}(\mathbb{R}^{n})}, \quad u\in C_{0}^{\infty}(\mathbb{R}^{n}).
	\end{equation}
	The constant $C_{\delta}$ depends on $\delta$ but not on $z$.
\end{proposition}

\begin{proof}
	\begin{equation}
	\begin{aligned}
	R_{z}(x, y) & =-[\eta_{\delta}(x, y), \Delta]F_{z}(x, y)\\
	& =(\Delta\eta_{\delta}(x, y))F_{z}(x, y)+2\nabla\eta_{\delta}(x, y)\cdot\nabla F_{z}(x, y)
	\end{aligned}.
	\end{equation}
	Here, all derivatives are with respect to $x$. $(\Delta\eta_{\delta}(x, y))F_{z}(x, y)$ can be handled in exactly the same way as in Proposition \ref{Prop1}, so we concentrate on $\nabla\eta_{\delta}(x, y)\cdot\nabla F_{z}(x, y)$.
	
	Taking Fourier transform, the operator $\frac{d}{dx_{i}}F_{z}(x, y)$ amounts to multiplying on the Fourier transform side by $\frac{\xi_{i}}{|\xi|^{2}-z}$. By dilation again, we may assume that $z\in\mathbb{S}^{2}\backslash\{1\}$:
	\begin{gather*}
	\Big\|\mathscr{F}^{-1}\Big\{\frac{\xi}{|\xi|^{2}-z/|z|}\cdot\hat{f}(\xi)\Big\}\Big\|_{L^{q}(\mathbb{R}^{n})}\leqslant A\|f\|_{L^{p}(\mathbb{R}^{n})}\\
	\Updownarrow\\
	\Big\|\mathscr{F}^{-1}\Big\{\frac{\xi}{|\xi|^{2}-z}\cdot\hat{f}(\xi)\Big\}\Big\|_{L^{q}(\mathbb{R}^{n})}\leqslant A|z|^{\frac{n}{2}(\frac{1}{p}-\frac{1}{q})-\frac{1}{2}}\|f\|_{L^{p}(\mathbb{R}^{n})}.
	\end{gather*}
	Note that this time, the power on $|z|$ is $\frac{1}{2}$ more than that in $\kappa_{p, q}$ because of the additional $\xi_{i}$. With this reduction, we can do as in Proposition \ref{Prop1}, making the following modification as we proceed.
	
	When dealing with $m_{2}(\xi)$ (the estimate becomes $|\partial^{\alpha}_{\xi}m_{2}(\xi)|\leqslant C_{\alpha, \epsilon_{0}}|\xi|^{-|\alpha|-1}$), note that the kernel of the corresponding operator is to be multiplied by a function supported in $\{(x, y): \frac{\delta}{2}\leqslant |x-y|\leqslant \delta\}$ (away from the diagonal $x=y$). Because of this, we need not decompose using a Littlewood-Paley bump function, but can apply integration by parts directly to $\int m_{2}(\xi)e^{i(x-y)\cdot\xi}d\xi$ to see that this kernel is $O((1+|x-y|)^{-N})$ for any natural number $N$ after multiplied by the function related to $\eta_{\delta}$. Hence, the part of $R_{z}$ corresponding to $m_{2}(\xi)$ is uniformly bounded from $L^{p}(\mathbb{R}^{n})$ to $L^{q}(\mathbb{R}^{n})$ for any $1\leqslant p\leqslant q\leqslant\infty$.
	
	For $m_{3}(\xi)$, when $z$ is near $1$, we observe that the additional $\xi_{i}$ can be absorbed into the $C_{0}^{\infty}$ function $\chi_{0}(\xi)$ in the treatment of $m_{3}(\xi)$ in Proposition \ref{Prop1}. Then the proof for Proposition \ref{Prop1} goes through without any difference, and therefore, we obtain the desired conclusion. Note that the bound $C$ depends on $\delta$ here because of the derivatives on $\eta_{\delta}$.
\end{proof}

Now we can carry out the ideas in Blair-Sire-Sogge \cite{BSS}. Recall that we are dealing with the $q\neq\infty$ cases.

Choose a lattice $\{Q_{j}\}$ of (nonoverlapping) cubes of side length $\delta$. For each $Q_{j}$, consider further the concentric cube $Q_{j}^{\ast}$ of radius $2\delta$. The $Q_{j}^{\ast}$'s have finite overlap, hence \begin{equation}
\sum\limits_{j}\mathbbm{1}_{Q_{j}^{\ast}}\leqslant C_{n},
\end{equation}
where $C_{n}$ is a constant depending only on $n$.

We have by \eqref{id}, \begin{equation} \label{idV}
u=H_{z}((-\Delta+V-z)u)+R_{z}u-H_{z}(Vu).
\end{equation}
Recall that $H_{z}(x, y)$ is supported in the set $\{(x, y): |x-y|\leqslant \delta\}$ and that $R_{z}(x, y)$ in the set $\{(x, y): \frac{\delta}{2}\leqslant |x-y|\leqslant \delta\}$. Taking $L^{q}$ norm on both sides of \eqref{idV} where $q$ is as in the main theorem and applying Propositions \ref{Prop1} and \ref{Prop2}, we have for $p$, $q$ as before and $z\in\mathcal{Z}_{p, q}$, \begin{equation} \label{Lq}
\|u\|_{L^{q}(Q_{j})}\leqslant \kappa_{p, q}(z)(C\|(-\Delta+V-z)u\|_{L^{p}(Q_{j}^{\ast})}+C_{\delta}|z|^{\frac{1}{2}}\|u\|_{L^{p}(Q_{j}^{\ast})})+C\|Vu\|_{L^{p}(Q_{j}^{\ast})}.
\end{equation}
In the above, we used the fact that $|z|\geqslant 1$ when $z\in\mathcal{Z}_{p, q}$ and the fact that $\kappa_{p, q}(z)\leqslant 1$ (hence there is no $\kappa_{p, q}(z)$ before $\|Vu\|_{L^{p}(Q_{j}^{\ast})}$). Raising both sides to the $q$-th power, we have \begin{equation} \label{qth}
	\|u\|_{L^{q}(Q_{j})}^{q}\leqslant \kappa_{p, q}(z)^{q}(C\|(-\Delta+V-z)u\|_{L^{p}(Q_{j}^{\ast})}^{q}+C_{\delta}(|z|^{\frac{1}{2}}\|u\|_{L^{p}(Q_{j}^{\ast})})^{q}+C\|Vu\|_{L^{p}(Q_{j}^{\ast})}^{q}.
\end{equation}
Since $V\in(L^{\sigma}(\mathbb{R}^{n})+L^{\infty}(\mathbb{R}^{n}))$, we can choose a $\delta$ so small that $C_{n}C\|V\|_{L^{\sigma}(Q^{\ast}_{j})}^{q}<\frac{1}{2}$ for any $j$. Here $C$ is the constant in \eqref{qth}.
Recalling that $\frac{1}{p}=\sigma+\frac{1}{q}$, we have by Holder's inequality \begin{equation} \label{Holder}
\|Vu\|_{L^{p}(Q_{j}^{\ast})}\leqslant\|V\|_{L^{\sigma}(Q_{j}^{\ast})}\|u\|_{L^{q}(Q_{j}^{\ast})}.
\end{equation}

Finally we sum up the balls $Q_{j}$, and use inequalities \eqref{qth}, \eqref{Holder} and our choice of $\delta$ to get estimates on $\mathbb{R}^{n}$: \begin{equation} \label{sum} \begin{aligned}
& \|u\|_{L^{q}(\mathbb{R}^{n})}^{q} \leqslant\sum_{j} \|u\|_{L^{q}(Q_{j})}^{q}\\
& \leqslant \sum_{j}\{\kappa_{p, q}(z)^{q}[C\|(-\Delta+V-z)u\|_{L^{p}(Q_{j}^{\ast})}^{q}+C_{\delta}(|z|^{\frac{1}{2}}\|u\|_{L^{p}(Q_{j}^{\ast})})^{q}]+C\|Vu\|_{L^{p}(Q_{j}^{\ast})}^{q}\}\\
& \leqslant \sum_{j} \{\kappa_{p, q}(z)^{q}[C\|(-\Delta+V-z)u\|_{L^{p}(Q_{j}^{\ast})}^{q}+C_{\delta}(|z|^{\frac{1}{2}}\|u\|_{L^{p}(Q_{j}^{\ast})})^{q}]\\
& \quad +C\|V\|_{L^{\sigma}(Q_{j}^{\ast})}^{q}\|u\|_{L^{q}(B_{j}^{\ast})}^{q}\}\\
& \leqslant \kappa_{p, q}(z)^{q}[C_{n}^{\frac{q}{p}}C\|(-\Delta+V-z)u\|_{L^{p}(\mathbb{R}^{n})}^{q}+C_{n}^{\frac{q}{p}}C_{\delta}(|z|^{\frac{1}{2}}\|u\|_{L^{p}(\mathbb{R}^{n})})^{q}]\\
& \quad +C_{n}C(\mathrm{sup}_{j}\|V\|_{L^{\sigma}(Q_{j}^{\ast})}^{q})\|u\|_{L^{q}(\mathbb{R}^{n})}^{q}\\
& \leqslant \kappa_{p, q}(z)^{q}[C\|(-\Delta+V-z)u\|_{L^{p}(\mathbb{R}^{n})}^{q}+C_{\delta}(|z|^{\frac{1}{2}}\|u\|_{L^{p}(\mathbb{R}^{n})})^{q}]+\frac{1}{2}\|u\|_{L^{q}(\mathbb{R}^{n})}^{q}.
\end{aligned}
\end{equation}
Note that we can sum up the $Q_{j}^{\ast}$ on the right in the above process just as we do on the left because $q>p$.

Moving the $\frac{1}{2}\|u\|_{L^{q}(\mathbb{R}^{n})}^{q}$ on the right of \eqref{sum} to the left and then taking $\frac{1}{q}$-th power on both sides gives the desired result.

Now we turn to the proof of the $q=\infty$ case, when the above ``summing up'' procedure does not work. We get around this difficulty by using test functions and duality. Here instead of using identity \eqref{id}, we use identity \eqref{1}:
\begin{equation} \label{1'}
	I=(-\Delta-z)\circ H_{z}+R_{z}.
\end{equation}

To prove \eqref{ResV}, it suffices to prove \begin{equation}
|\int u\psi dx|\leqslant C\kappa_{p, \infty}(z)(\|(-\Delta+V-z)u\|_{L^{p}(\mathbb{R}^{n})}+|z|^{\frac{1}{2}}\|u\|_{L^{p}(\mathbb{R}^{n})})+\frac{1}{2}\|u\|_{L^{\infty}(\mathbb{R}^{n})},
\end{equation}
for any $u\in L^{\infty}(\mathbb{R}^{n})$ and $\psi\in C_{0}^{\infty}(\mathbb{R}^{n})$ satisfying $\|\psi\|_{L^{1}(\mathbb{R}^{n})}=1$. Applying identity \eqref{1'}, we have
\begin{equation} \begin{aligned}
|\int u\psi dx| & \leqslant |\int u((-\Delta-z)\circ H_{z}\psi) dx|+|\int u(R_{z}\psi) dx|\\
&\leqslant |\int ((-\Delta-z)u)(H_{z}\psi) dx|+|\int u(R_{z}\psi) dx|\\
& \leqslant |\int ((-\Delta+V-z)u)(H_{z}\psi) dx|+|\int u(R_{z}\psi) dx|+|\int uV(H_{z}\psi) dx|.
\end{aligned}
\end{equation}

By duality, $\|H_{z}\|_{L^{1}(\mathbb{R}^{n})\rightarrow L^{p'}(\mathbb{R}^{n})}\leqslant C\kappa_{p, \infty}(z)$, hence from Holder's inequality we get \begin{equation}
\begin{aligned}
|\int ((-\Delta+V-z)u)(H_{z}\psi) dx| & \leqslant \|(-\Delta+V-z)u\|_{L^{p}(\mathbb{R}^{n})}\|H_{z}\psi\|_{L^{p'}(\mathbb{R}^{n})}\\
&\leqslant C\kappa_{p, \infty}(z)\|(-\Delta+V-z)u\|_{L^{p}(\mathbb{R}^{n})}\|\psi\|_{L^{1}(\mathbb{R}^{n})}\\
& \leqslant C\kappa_{p, \infty}(z)\|(-\Delta+V-z)u\|_{L^{p}(\mathbb{R}^{n})}.\\
\end{aligned}
\end{equation}
For the same reason, since $\|R_{z}\|_{L^{1}(\mathbb{R}^{n})\rightarrow L^{p'}(\mathbb{R}^{n})}\leqslant C\kappa_{p, \infty}(z)|z|^{\frac{1}{2}}$, \begin{multline}
|\int u(R_{z}\psi) dx|\leqslant \|u\|_{L^{p}(\mathbb{R}^{n})}\|R_{z}\psi\|_{L^{p'}(\mathbb{R}^{n})}\leqslant \\ C\kappa_{p, \infty}(z)|z|^{\frac{1}{2}}\|u\|_{L^{p}(\mathbb{R}^{n})}\|\psi\|_{L^{1}(\mathbb{R}^{n})}
\leqslant C\kappa_{p, \infty}|z|^{\frac{1}{2}}\|u\|_{L^{p}(\mathbb{R}^{n})}.
\end{multline}

It remains to tackle $|\int uV(H_{z}\psi) dx|$. By the definition of the Kato class and the well-known estimate of $H_{z}(x, y)$ that $|H_{z}(x, y)|\leqslant Ch_{n}(|x-y|)$ when $|x-y|$ is small (see \cite{BSS} for detail),
\[H_{z}(Vu)=\int H_{z}(x, y)V(y)u(y)dy, \quad u\in L^{\infty}(\mathbb{R}^{n})\]
defines an absolutely convergent integral and is bounded in $x$. It then follows easily that $\int uV(H_{z}\psi) dx$ is absolutely convergent as well.
Therefore, we may use Fubini's theorem to get \begin{equation}
|\int uV(H_{z}\psi) dx|=|\int H_{z}(Vu)\psi dx|\leqslant\|H_{z}(Vu)\|_{L^{\infty}(\mathbb{R}^{n})}\|\psi\|_{L^{1}(\mathbb{R}^{n})}=\|H_{z}(Vu)\|_{L^{\infty}(\mathbb{R}^{n})}.
\end{equation}
The question now becomes bounding $\|H_{z}(Vu)\|_{L^{\infty}(\mathbb{R}^{n})}$ by $\frac{1}{2}\|u\|_{L^{\infty}(\mathbb{R}^{n})}$. For this, we simply apply the definition of Kato class and the well-known estimate of $H_{z}(x, y)$ above again to see that if $\delta$ is chosen small enough, then we will have \[|\int H_{z}(x, y)V(y)dy|\leqslant\frac{1}{2}.\]
This finishes the proof of the Theorem.
\qed

\paragraph{Proofs of Corollaries}
Corollaries \ref{Cor1} and \ref{Cor2} are obvious. To prove Corollary \ref{Cor3}, we essentially invoke the same idea as that in the proof of the main theorems. However, this time, we use a simpler relation: \begin{equation}
u=(-\Delta-(\lambda+i\mu)^{2})^{-1}(-\Delta+V-(\lambda+i\mu)^{2})u-(-\Delta-(\lambda+i\mu)^{2})^{-1}(Vu).
\end{equation}
Take $L^{q}$ norms on both sides and apply Kwon-Lee \cite{KL}'s resolvent estimates for the operator ``$(-\Delta-(\lambda+i\mu)^{2})^{-1}$'' to get \begin{equation}
	\|u\|_{L^{q}(\mathbb{R}^{n})}\leqslant C\lambda^{n(\frac{1}{p}-\frac{1}{q})-2+\gamma_{p, q}}\mu^{-\gamma_{p, q}}\|(-\Delta+V-(\lambda+i\mu)^{2})u\|_{L^{p}(\mathbb{R}^{n})}+C\|Vu\|_{L^{p}(\mathbb{R}^{n})}.
\end{equation}
For $\|Vu\|_{L^{p}(\mathbb{R}^{n})}$, use Holder's inequality as in the proof of the main theorems. Recalling our condition that $\|V\|_{L^{\sigma}(\mathbb{R}^{n})}$ is small enough, we bound $C\|Vu\|_{L^{p}(\mathbb{R}^{n})}$ by $\frac{1}{2}\|u\|_{L^{p}(\mathbb{R}^{n})}$. The conclusion of the corollary then follows. Observe that here we don't break $\mathbb{R}^{n}$ into small cubes.

The first Statement in Corollary \ref{Cor4} is obvious. To show the second Statement, simply set $\lambda=2\mu$ and substitute $\chi_{[0, 2\mu)}^{V}f$ for $u$ in \eqref{Extension2} (the condition that $\mu$ be bounded away from $0$ is just to ensure that $2\mu+i\mu\in\mathcal{Z}_{p, q}$): \begin{equation}
\|\chi_{[0, 2\mu)}^{V}f\|_{L^{q}(\mathbb{R}^{n})}\leqslant C\mu^{n(\frac{1}{p}-\frac{1}{q})-2}\|(-\Delta+V-(\lambda+i\mu)^{2})\chi_{[0, 2\mu)}^{V}f\|_{L^{p}(\mathbb{R}^{n})}.
\end{equation}
Noticing that if $\tau\in[0, 2\mu)$, then $|\tau^{2}-(2\mu+i\mu)^{2}|\approx\mu^{2}$, we see that the desired conclusion follows from the spectral theorem. \qed

\section*{Appendix}
We state Bourgain's interpolation method.
\begin{lemma}
	Suppose that an operator $T$ between function spaces is the sum of the operators ${T_{j}}$: \[T=\displaystyle\sum_{j=1}^{\infty}T_{j}.\] If for $1 \leqslant p_{1}, p_{2}, q_{1}, q_{2} \leqslant \infty$, there exist $\beta_1,\beta_2 > 0$ and $M_{1}, M_{2} > 0$ such that each $T_{j}$ satisfies
	\[\|T_j\|_{L^{p_1}\rightarrow L^{q_1}}\le M_12^{-j\beta_1},\]
	and
	\[\|T_j\|_{L^{p_2}\rightarrow L^{q_2}}\le M_22^{j\beta_2},\]
	then we have restricted weak type estimate for the operator $T$ between two intermediate spaces:\[
	\|Tf\|_{L^{q, \infty}}\leqslant C(\beta_{1}, \beta_{2}) M_{1}^{1-\theta}M_{2}^{\theta}\|f\|_{L^{p, 1}},\] where
	\[\theta=\frac{\beta_{1}}{\beta_{1}+\beta_{2}},\]
	\[\frac{1}{p}=\frac{1-\theta}{p_{1}}+\frac{\theta}{p_{2}}, \ \  \frac{1}{q}=\frac{1-\theta}{q_{1}}+\frac{\theta}{q_{2}}.\]
\end{lemma}
If the power of 2 in the bounds changes linearly as we go from $(\frac{1}{p_{1}}, \frac{1}{q_{1}})$ to $(\frac{1}{p_{2}}, \frac{1}{q_{2}})$, then the $(\frac{1}{p}, \frac{1}{q})$ in the conclusion is exactly the point at which the power becomes $0$.

\paragraph{Acknowledgement} The author is grateful to his postdoc mentor Professor Baoping Liu for his support, encouragement and advice throughout the author's time at BICMR, Peking University. The author is also indebted to Cheng Zhang of the University of Rochester for reading the manuscript and proposing many suggestions that have improved the article a great deal. Finally, the author would like to thank the anonymous referee for advising him to add the application of the main theorems to the article. This has made the article much better.

\bibliography{Resol.with.Potentials}

\bibliographystyle{plain}
\vspace{0.5cm}
\noindent\textit{Email address:} rentianyi@pku.edu.cn

\noindent Beijing International Center for Mathematical Research, Peking University, Beijing, China 100871

\end{document}